\documentclass[a4paper,english,oneside,12pt]{amsart}
\usepackage[latin1]{inputenc}
\usepackage{amsmath}
\usepackage{amssymb}
\usepackage{amsfonts}
\usepackage{latexsym}
\usepackage{mathtools}
\usepackage{amscd}
\usepackage{xy} \xyoption{all}
\usepackage{hyperref}
\usepackage{enumerate}

\usepackage[normalem]{ulem}

\usepackage{color}

\newcommand{\blue}[1]{#1}

\newcommand{\TOKEok}[1]{}

\newcommand{\N}{\mathbb{N}}
\newcommand{\Z}{\mathbb{Z}}
\newcommand{\R}{\mathbb{R}}
\newcommand{\T}{\mathbb{T}}

\newcommand{\maptorus}[1]{\mathsf{S}{#1}}

\newcommand{\hp}{\textnormal{Homeo}_+}
\newcommand{\hpiso}{ \textnormal{Homeo}_+^\textnormal{iso} }
\newcommand{\hporb}{ \textnormal{Homeo}_+^{\textnormal{orb}} }

\renewcommand{\phi}{\varphi}

\newtheorem{lemma}{Lemma}[section]

\newtheorem{theorem}{Theorem}[section]
\newtheorem{prop}{Proposition}[section]
\newtheorem{proposition}{Proposition}[section]

\theoremstyle{definition} 

\newtheorem{definition}{Definition}[section]

\newtheorem{example}{Example}[section]
\newtheorem{remark}{Remark}[section]

\input xy
\xyoption{all}

\title{Flow equivalence and isotopy for subshifts}

\author{Mike Boyle}
\author{Toke Meier Carlsen}
\author{S\o ren Eilers}
\date{\today}

\numberwithin{equation}{section}

\begin{document} 
\begin{abstract}
We study basic properties of flow equivalence on one-dimensional compact metric spaces with a particular emphasis on isotopy in the group of (self-) flow equivalences on such a space. In particular, we show that an orbit-preserving such map is not always an isotopy, but that this always is the case for suspension flows of irreducible shifts of finite type. We also provide a version of the fundamental discretization result of Parry and Sullivan which does not require that the flow maps are either injective or surjective. Our work is motivated by applications in the classification theory of sofic shift spaces, but has been formulated to supply a solid
and accessible foundation for other purposes.
\end{abstract}

\maketitle

\tableofcontents

\section{Introduction}

In this paper we set out some basic results
around flow equivalence and isotopy
involving flows \blue{without fixed points} on one-dimensional
compact \blue{metric} spaces 
and the return maps to their
cross sections. 
Our motivation is to provide a solid
and accessible foundation for other work.

The study of flow equivalence of shifts of
finite type (SFTs) is very well understood and has had profound applications to
$C^*$-algebras. We use several results in our work on
flow equivalence within certain subshift
classes \cite{bce:gfe,bce:sofic}.
Flow equivalence of $G$ shifts {of} finite type
has provided some information about knot invariants. 

 We are also interested in the
``mapping class group'' of a subshift,
especially  of an irreducible shift of
finite type. This is the group of
self-homeomorphisms of the mapping torus
of that subshift up to isotopy,
studied in \cite{mb:fesftpf, Sompong}.

By a flow map, we mean a continuous map between
spaces with flows, which maps each domain orbit
onto some range orbit by an orientation preserving
local homeomorphism. In {Section}  \ref{sec:flows},
we give basic background on flows and cross sections. 
In Section 
\ref{sec:isotopy} we study when a flow equivalence
mapping each orbit into itself is isotopic 
to the identity within the group of flow equivalences
of a space $Y$ to itself.
Even for subshifts, this can be tricky.
Suppose $Y$ is the mapping torus of a subshift $X$  and
$f:Y\to Y$ is a flow equivalence mapping each orbit
into itself. Must $f$ be isotopic to the identity?
\blue{The answer is yes if $X$ is a minimal shift 
\cite[Theorem 2.5]{app:shgtl} or 
if $X$ is an irreducible shift of finite type
 (Theorem \ref{isftorbittoiso}); but for a 
 reducible shift of finite type or mixing sofic shift,
the answer is no (Examples \ref{badsftexample},\ref{badsoficexample}).}  
The main  criterion for this isotopic triviality
is given in Theorem \ref{isoflowprop};
it should be known, but we haven't found 
\blue{its statement in the literature, despite the abundance of 
related results.}

In {Section} \ref{sec:ps}, we give a formulation and extension
of the key argument of the Parry-Sullivan paper
 \cite{parrysullivan}
which is the basis for connecting the dynamics of
one-dimensional flows and the discrete systems given
by return maps to cross sections. In
particular,  
 we give a version applicable to 
flow maps which are neither surjective
nor injective (which we need in 
\cite{bce:sofic} to study flow equivalence of sofic
shifts via their canonical SFTs covers{)}. 
In {Section} 
\ref{sec:flowcodes}, we introduce flow codes,
which play for flow equivalence of subshifts
the role block codes
play for homomorphisms  of subshifts.

In Section \ref{floweqconjsec},
given a flow equivalence of irreducible SFTs respecting
lengths of finite orbits, we show it is induced  by a
conjugacy, and use this to prove Theorem \ref{isftorbittoiso}.  
Also, given a flow equivalence
of systems $Y,Y'$ with cross sections $C,C'$, 
we characterize when a flow equivalence $Y\to Y'$
can be lifted to an equivalence 
$C \times \R \to C' \times \R$ of their covering spaces. 

In Section \ref{sec:extension}, for flows on
one-dimensional spaces we prove two extension
results. An isotopically trivial map on a subflow
can be extended to an isotopically trivial map on the
entire flow. A cross section of a subflow can be
extended to a cross section of the entire flow. 

Some results are a
stripped down version  of ideas and results from the  
theory of smooth flows on hyperbolic sets 
(see \cite[Sections 2.2, 2.9 and 
19.2]{hk:book}), as we indicate.   
{We have given independent proofs
  for these results, because the
   smooth statements don't include the zero-dimensional case;
  some smooth arguments do not translate
  mechanically to the zero-dimensional setting;
  and some arguments adequate for dimension zero are much shorter
  and easier.} 

\blue{The one-dimensional spaces $Y$ we study can be considered as 
(a quite special class  of) tiling spaces. The
 large literature on tiling spaces contains results 
(see e.g. in \cite{JS}) 
 which specialize to imply some of our statements in cases,  
e.g. when the 
suspension flow on $Y$ is minimal.}

{\subsection*{Notation}
When $L_1$ and $L_2$ are two sets of words, then we let $L_1L_2$ denote the set $\{uv:u\in L_1,\ v\in L_2\}$.
When $L$ is a set of words, then we let $L^*$ denote the set of words that are concatenations of zero or more elements from $L$ (so the empty word is an element of $L^*$).}

{\subsection*{Acknowledgements}
We thank Sompong Chuysurichay for valuable comments}
\blue{on} {an earlier draft of this paper.} {This work was supported by the Danish National Research Foundation through the Centre for Symmetry and Deformation (DNRF92), and by VILLUM FONDEN through the network for Experimental Mathematics in Number Theory, Operator Algebras, and Topology.}

\section{Flows and cross sections} \label{sec:flows}

{In this section, we give basic background on flows and cross sections.}

\begin{definition}\label{flowdef} 
A {\it flow} in this paper is a continuous and {\it fixed point free}\footnote{{For flows with fixed points, flow equivalence (as in {Definitions} \ref{fenotation}) is a much weaker relation than for flows {without} fixed points (see e.g. \cite{MarcusReparam}), and compactness arguments are much less effective.}} action of $\R$ on a nonempty compact metric\footnote{In this paper, the choice of metric compatible with the topology won't matter.}
space. A flow on a compact metric space $Y$ is given by a continuous map $\gamma: Y\times \R\to  Y$ such that for all $s,t$ in $\R$ and $y$ in $Y$, $\gamma (\gamma (y,s),t) = \gamma (y, s+t)$, and $\gamma (y,0) = y$. For $t$ in $\R$, $y\mapsto \gamma (y,t)$ defines the time $t$ homeomorphism $\gamma_t : Y\to Y$. 
\end{definition}  

\begin{definition} \label{mapdef} 
{Given a compact metric space $X$,  
a homeomorphism $T: X\to X$, and 
a continuous function $f : X\to (0, \infty)$}, 
let $Y$ be the quotient of $\{ (x,t): 0\leq t\leq f(x)\}$ 
by the identifications $(x,f(x)) \sim (T(x),0)$ for 
all $x$ in $X$. 
{The map $\gamma_t : ((x,s),t) \mapsto (x,s+t)$ is a flow on $X\times \R$, which commutes with the $\Z$ action generated by $(x,t)\mapsto (T(x),t-f(x))$. The space $Y$ can be presented as the orbit space of this $\Z$ action. The flow on $X\times \R$ 
induces a flow on $Y$. This is one construction {of the} {\it flow under a function}. This presentation generalizes to other groups.}

{The space} $X$ is the {\it base} 
and $f$ is the {\it ceiling function}. 
In the case that $f$ is the 
constant function 1, $Y$ is the {\it mapping torus 
of $T$}, which (abusing notation) we denote 
$\maptorus{X}$. {The induced flow on $\maptorus{X}$ is the {\it suspension of $T$}}.
\end{definition} 

\begin{definition}\label{crosssecdefn} 
\cite[Sec. 7]{Schwartzman57}
Let $\gamma : Y \times \R \to Y$ be a flow as in Definition
{\ref{flowdef}}. A {\it cross section} to the flow is a closed 
subset $C$ of $Y$ such that the restriction of $\gamma$ to 
$C\times \R \to Y$ is a surjective local homeomorphism. 
 \end{definition}  

\begin{remark} \label{equivtocross} 
For a flow $\gamma$ on $Y$, it
 is not difficult to check that the following are equivalent 
conditions  on a closed subset $C$ of $Y$.  
\begin{enumerate} 
\item $C$ is a cross section to the flow. 
\item 
$\gamma : C\times \R \to Y$ is surjective, 
and there exists $\delta >0$ such that 
$\gamma : C \times (-\delta, \delta )\to {\gamma(C \times (-\delta, \delta ))}$ is a local homeomorphism. 
\item 
$\gamma : C\times \R \to Y$ is surjective, 
and there exists $\delta >0$ such that 
$\gamma : C \times (-\delta, \delta )\to {\gamma(C \times (-\delta, \delta ))}$ is a  homeomorphism. 
\item $\gamma : C\times \R \to Y$ is surjective, 
and there is a well defined continuous return time function 
$\tau_C : C \to \R$, given by
$\tau_C(x) = \min \{t>0: \gamma_t(x) \in C\}$.
\item $\gamma : C\times \R \to Y$ is surjective, 
and there is a well defined return time function 
$\tau_C : C \to \R$, given by
$\tau_C(x) = \min \{t>0: \gamma_t(x) \in C\}$, which is bounded away from 0. 
\end{enumerate} 
\end{remark}

For the flow under a function in Definitions
\ref{mapdef}, the image of $X\times \{ 0\}$ in $Y$ 
is a cross section for the flow. Conversely, 
by Remark \ref{equivtocross}, if $C$ 
is a cross section to a flow $\gamma $ on $Y$, then 
the flow $\gamma $ is topologically conjugate 
(as defined in {Definition} \ref{fenotation}) to the flow under 
a function built with base $C$ and ceiling function the 
return time function on $C$.

Large classes of flows will not admit a cross section 
(see e.g. \cite{Reeb52, Schwartzman57, Fried82cross}), 
but a flow on a one{-}dimensional 
space will always admit a cross section. 
This fact,  stated in Proposition \ref{zerocross} below, 
is the analogue for continuous flows on compact 
one{-}dimensional metric spaces of the Ambrose-Kakutani 
Theorem \cite{Ambrose1941,
AmbroseKakutani1942} 
for aperiodic measure preserving flows on a 
standard probability space.  

We will be using the following basic
tool for studying flows.

\begin{theorem} \label{localcross} \cite[Thm. V.2.15]{ns:book} 
  Suppose $\gamma: Y\times \R \to Y$  is a flow on a 
   compact metric space $Y$,  
 $p\in  Y$, $T>0$ and $|t| \leq 4T \implies
  \gamma (p,t) \neq p$.

  Then there exists a closed set $F$
  in $Y$ such that $\gamma $ maps $F\times [-T,T]$ homeomorphically
  onto a neighborhood of $p$.
\end{theorem}

\begin{remark}
  If $F$ is a closed set in $Y$, $J$ is an interval
  and $\gamma$ is injective on $W= F\times J$ and 
  $\gamma (W)$ has nonempty interior, 
  then
 $F$ is called a {\it local section} (or {\it local cross section}) and 
  $\gamma (W)$ is called a {\it flowbox}
  (a flowbox neighborhood for points in its interior).
  In  settings with more regularity (differentiable,
  Lipschitz, {\dots}) more {conditions} might be demanded
  of the flow in the flowbox.
\end{remark} 
\begin{remark} 
The statement of 
\cite[Thm. V.2.15]{ns:book}   does not quite cover the
statement of Theorem \ref{localcross}. However, the proof
of \cite[Thm. V.2.15]{ns:book} 
finds for arbitrarily small $\tau>0$ a closed set $F$ in $Y$ 
such that $\gamma (F\times [-2\tau, 2\tau ])$ contains 
a neighborhood of $p$ 
and $\gamma $ is injective on $F\times [-T,T]$. This
neighborhood contains
{$\gamma (F\times [-(T-4\tau ),T-4\tau])$}. Beginning with
some $T+\epsilon $ in place of $T$ we get Theorem \ref{localcross}. 
\end{remark}

Proposition \ref{zerocross} below is  well known, 
and is a special case of more general results 
\cite{AartsMartens1988, KeynesSears1979}, but for completeness we will 
include a 
short 
proof.  
By ``dimension'', we mean  
covering dimension (but for compact metric spaces, various 
standard conditions agree \cite{HWbook}). We let $B_{\delta}(y)$ denote 
the open ball with radius $\delta $ and center $y$. 

\begin{proposition}\label{zerocross}
 Suppose $\gamma $ is a flow on a 
one-dimensional compact metric space $Y$. Then the flow 
has a cross section, and every cross section of the flow is 
zero-dimensional. 
\end{proposition}

\begin{proof}
  Pick $T>0$ with $4T$ smaller than the period of any point. 
  Given $y\in Y$, let $U= \gamma (F \times [-T,T])$
  be a flowbox neighborhod of $y$ as in Theorem \ref{localcross}.   
Being an injective map between 
compact Hausdorff spaces,
the restriction of $\gamma $ to
$F \times [-T,T]$ 
 is a homeomorphism onto 
$U$.  Because $Y$ is one-dimensional, 
it follows that $F$ is zero-dimensional\footnote{
Already in 
  \cite{Hurewicz1930}, 
Hurewicz proved { that a 
product of $n$ one{-}dimensional compact metric
  spaces} has dimension 
 at least $n$.}.   Take $C$ relatively clopen in $F$
such that $\gamma (C\times (-T, T ))$ 
is an open neighborhood  $U$ of $y$. 

Let $\{ U_i\} = \{ {\gamma (C_i \times (-T, T))} : 1\leq i
\leq N\}$ be a finite collection of such sets whose union covers $Y$. 
Given $j\neq 1$ and $p$ in $C_1\cap C_j$, because $U_j$ is open 
there will be a relative neighborhood $W_{p,j}$ of $p$ in $C_1$ such
that $W_{p,j}
\subset U_j$. We take $W_{p,j} $ clopen in $C_1$; by compactness, 
a finite union $W_j$ of such sets covers $C_1\cap C_j$. 
Replace 
$C_1$ with $C_1 \setminus \cup_{j>1} W_j$. Now 
$C_1 $ is disjoint from the other $C_j$; for a small $T_1>0$,  
$\gamma (C_1 \times (-T_1,T_1 ))$ is still open in $Y$
and is disjoint from the sets $\gamma (C_j \times (-T_1,T_1 ))$, 
and  every orbit still intersects $\cup_i C_i$.
Iterating this move,
we find $\epsilon = T_N>0$ 
and compact disjoint zero-dimensional sets $C_1, \dots , C_N$ such that 
the open sets 
$U_j = \gamma (C_j \times (-\epsilon, \epsilon ))$ are disjoint
and every orbit hits some $C_j$. 
Let $C = \cup_j C_j$. Then 
$\gamma : C\times \R \to Y$ is  surjective
and $\gamma : C\times (-\epsilon, \epsilon)  \to Y$ 
is a  homeomorphism. Therefore $C$ is a cross section. 
\end{proof} 

\begin{definition}\label{fenotation} 
A {\it homomorphism} of flows $(Y_1,\gamma_1) \to 
(Y_2,\gamma_2)$ 
is a continuous map 
$h: Y_1 \to Y_2$  such that 
$h(\gamma_1 (y,t)) = \gamma_2 (h(y),t)$ 
 for all $t\in \R $ and 
all $y\in Y_1$.   
An  {\it epimorphism} (or {\it semiconjugacy}) of flows is a 
surjective homomorphism of flows; an 
{\it isomorphism} (or {\it conjugacy} or {\it topological conjugacy}) of flows 
is a homomorphism of flows
defined by a homeomorphism\footnote{In our setting of 
compact metric spaces, a bijective 
homomorphism of flows must be a topological conjugacy.}.  

A {\it flow map } 
is a continuous map 
$h: Y_1 \to Y_2$  such that for every $y$ in $Y_1$, 
the restriction of {$h$} to the $\gamma_1$ orbit 
of $y$ is an orientation-preserving local homeomorphism onto the 
$\gamma_2$ orbit of $h(y)$. 
A {\it semiequivalence of flows} 
(or {\it flow semiequivalence})  is a 
surjective flow map. 
An {\it equivalence of flows} 
(or {\it flow equivalence})  is a 
semiequivalence of flows defined by a homeomorphism 
(in our compact metric setting, a bijective flow   map). 

By a {\it flow equivalence} of two homeomorphisms, we 
mean an equivalence of the suspension flows on their
mapping tori. 
\end{definition} 
For example, suppose flows on $Y_1, Y_2$ 
are built as flows under continuous positive
functions 
$\phi_1, \phi_2$ with the same base homeomorphism $T: X\to X$.   
Then there is a  flow equivalence $Y_1\to Y_2$ 
which is an extension of the identity
map between the bases $X\times \{0 \}$.  





%


\begin{remark}\label{fenotationremark} 
Our choice of terminology for equivalence 
in Definitions \ref{fenotation} 
follows \cite{jf:fesft} and \cite[Sec. 4.7]{Robinsonedition1}. 
It is 
well adapted to our topic 
of considering when two maps are flow equivalent 
(terminology from \cite{parrysullivan} and perhaps earlier): 
we naturally want to refer to a morphism by which two   
maps are flow equivalent as a flow equivalence. 
Caveat: various other terminologies have been used by 
different  authors (e.g. \cite{Smale1967,Irwin,hk:book}). For example,   ``topological conjugacy'' 
in this paper 
 is ``$C^0$ flow equivalence''  in \cite[Sec. 2.2]{hk:book}, 
and ``flow equivalence'' in this paper 
is 
``$C^0$ orbit equivalence'' 
 in \cite[Sec. 2.2]{hk:book}.  

With the term ``morphism'' of flows committed by our use of 
isomorphism of flows, we end up using ``flow map'' for the corresponding 
notion related to flow equivalence.  
\end{remark}

\begin{proposition} 
\label{equivcrossconditions}
Suppose $\pi: Y \to Y'
$  is {a flow map},  
with $C,C'$ subsets of $Y, Y'$ such that 
$C= \pi^{-1}(C')$. {If $\pi$ is surjective, then} the following are equivalent. 
\begin{enumerate} 
\item 
$C$ is a cross section of $Y$. 
\item 
$C'$ is a cross section of $Y'$. 
\end{enumerate} 
{In general (i.e., if $\pi$ is not assumed surjective), if}
$C' $ is a cross section 
for $Y'$,  then $C$ is a cross section for $Y$. 
\end{proposition}
\begin{proof} 

{
$(2)\implies (1):$ 
Because $C'$ is a cross section, $\pi$ is a flow map, and $\pi^{-1}(C')=C$, 
the following hold: $C$ is closed, every
orbit hits $C $, and the return time  
$\tau_C(x):=\min\{t>0:\gamma_t(x)\in C\}$ is well defined for every $x\in C$.
It remains to show that $\tau_C$ is bounded away from 0. Suppose not. Then there is a sequence $\{x_n\}$ in $C$ such that $\tau_C(x_n)\to 0$. It follows from the compactness of $C$ that there is a subsequence $\{x_{n_i}\}$ and an $x\in C$ such that $x_{n_i}\to x$. Let $\tau_{C'}$ be the return function for $C'$. Choose a $\delta>0$ such that $\tau_{C'}(x')>\delta$ for every $x'\in C'$. Since the restriction of $\pi$ to the $\gamma$ orbit of $x_{n_i}$ is an orientation-preserving local homeomorphism onto the $\gamma'$ orbit of $\pi(x_{n_i})$, it follows that there is a $t_i\in (0,\tau_C(x_{n_i}))$ such that $\pi(\gamma(x_{n_i},t_i))=\gamma'(\pi(x_{n_i}),\delta)$. Then $t_i\to 0$, so $\gamma'(\pi(x_{n_i}),\delta)=\pi(\gamma(x_{n_i},t_i))\to \pi(x)$, but that cannot be the case since $\tau_{C'}(x')>\delta$ for every $x'\in C'$. Hence $\tau_C$ is bounded away from 0, and $C$ is a cross section of $Y$.}

{$(1) \implies (2)$: 
Because $C$ is a cross section, $\pi$ is a flow map, and $\pi^{-1}(C') = C$, 
the following hold:  
 $C'$ is closed; every orbit hits $C'$; the return time
$\tau_{C'}(x'):=\min\{t>0:\gamma'_t(x')\in C'\}$ is well defined for each $x'\in C'$. It remains 
to show $\tau_{C'}$ is bounded away from 0. Suppose not. Then there is a sequence $\{x'_n\}$ in $C'$ such that $\tau_{C'}(x'_n)\to 0$. Choose for each $n$, an $x_n\in C$ such that $\pi(x_n)=x'_n$. It follows from the compactness of $C$ that there is a subsequence $\{x_{n_i}\}$ and an $x\in C$ such that $x_{n_i}\to x$. Let $\tau_C$ be the return function for $C$. Choose $\delta>0$ such that $\tau_C(y)>\delta$ for every $y\in C$. Since the restriction of $\pi$ to the $\gamma$ orbit of $x_{n_i}$ is an orientation-preserving local homeomorphism onto the $\gamma'$ orbit of $\pi(x_{n_i})$, it follows that there is a $t'_i\in (0,\tau_{C'}(x'_{n_i}))$ such that $\pi(\gamma(x_{n_i},\delta))=\gamma'(\pi(x_{n_i}),t'_i)$. Then $t'_i\to 0$. So $\pi(\gamma(x_{n_i},\delta))=\gamma'(\pi(x_{n_i}),t'_i)\to \pi(x)$. Thus $\pi(\gamma(x,\delta))= \pi(x)$, from which it follows that $\gamma(x,\delta)\in C$, but that cannot be the case since $\tau_C(y)>\delta$ for every $y\in C$. Hence $\tau_{C'}$ is bounded away from 0, and $C'$ is a cross section of $Y'$.}



{\it The nonsurjective case.} If $\pi$ is not surjective and 
 $C'$ is a cross section, then $C' \cap \pi (Y)$ is  a 
cross section for the restriction of the $Y'$ flow to 
$\pi (Y)$. Therefore the final claim follows from the 
case $(2)\implies (1)$. 
\end{proof}

\begin{definition}\label{discretedefn} Suppose $T:X\to X$ is a homeomorphism of 
a compact zero-dimensional  metric space. A {\it discrete cross section} for 
$T$ is a closed subset $C$ of $X$ with a continuous 
 function $r: C \to \N$ such that 
$r (x)  = \min \{ k\in \N: T^k(x) \in C\}$ and 
$X=\{ T^k(x): x\in C, k\in \N\}$. 
\end{definition} 

In Definition \ref{discretedefn}, the function $r$ must be
bounded and locally constant on $C$. Then $X$ is 
the disjoint union of finitely many clopen sets of the form 
$T^i(C_j)$, $0\leq i < j$, with $C_j=\{ x\in C: r(x)=j \}$. 
Consequently, if $K$ is a subset of 
a zero-dimensional cross section  $C$ to a flow, 
then $K$ is a cross section to the 
flow if and only if $K$ is a discrete cross section for the 
discrete system $(C, \rho_C )$. 

If $h:Y\to Y'$ is a flow map 
(perhaps an equivalence) and $C'$ is a cross section for 
$Y'$, 
then $C=h^{-1}(C')$ is a cross section for $Y
$ and the restriction {$h|_C$} defines a morphism 
of the return maps $\rho_C , \rho_{C'}$ {(i.e.,
$h|_C$ is a continuous map from $C$ into $C'$ which intertwines 
$\rho_C$ and $ \rho_{C'}$)}. Conversely, 
if $\phi : C\to C'$ is a morphism of return maps to cross sections $C,C'$ 
for $Y, Y'$, then $\phi$ 
 extends to a flow map $h:Y\to Y'$;   
in the case that $h^{-1} (C')=C$, 
we say this flow map is {\it induced by $\phi$.} 
If $h_1,h_2$
are two flow maps induced by 
a morphism $\phi$, then there is a flow equivalence $h_3: Y\to Y$, 
which is an extension of the identity map on $C$ and is isotopic to the
identity in the 
group $\hp (Y)$ 
of orientation-preserving orbit-preserving homeomorphisms 
on $Y$ (Definitions \ref{groupdefns}), 
such that $h_1 = h_2\circ h_3$.  
\section{Isotopy} \label{sec:isotopy}

{We will now study when a flow equivalence
mapping each orbit into itself is isotopic 
to the identity within the group of flow equivalences
of a space $Y$ to itself}. We begin with some definitions. 
\begin{definition} \label{groupdefns} 
Suppose $Y$ is a compact metric  space with a fixed-point free 
continuous $\R$-action (a flow). 
\begin{enumerate} 
\item 
$\hp (Y)$ is  the group of
homeomorphisms 
of $Y$ which map flow orbits onto flow orbits, preserving the 
orientation 
given by the flow direction. This is the group of 
{self-equivalences} of the flow
on 
$Y$. 
\item 
$\hporb (Y)$ is  the 
group of 
homeomorphisms in $\hp (Y)$ 
 which map
 each flow orbit to itself. 

\item 
$\hpiso (Y)$ 
is the subgroup of $\hporb (Y)$ 
consisting of 
 the homeomorphisms isotopic in $\hporb (Y)$ 
to the identity.
\blue{\footnote{\blue{By a theorem of Aliste-Prieto and Petite
\cite{app:shgtl}, $\hpiso (Y)$ is known to be a simple group.}}}
 \\  
 (In detail:  $h\in \hpiso (Y)$ if there is a continuous 
map $H: Y\times [0,1] \to Y$ such that, with $h_t(y)= H(y,t)$, 
each $h_t \in \hp(Y)$, $h_1=h $ and $h_0=\text{Id}$.) \\
Equivalently, 
$\hpiso (Y)$ is the path component of the identity in 
$\hp (Y)$ (with $\hp (Y)$  topologized by a metric 
$\text{dist}(f,g) = \max \{
 {d}(f(y),g(y)) +
    {d}(f^{-1}(y),g^{-1}(y)) : 
y\in Y\}$, with $d$ a metric on $Y$  compatible 
with the topology). 
\end{enumerate} 
\end{definition} 

Homeomorphisms $\phi_0, \phi_1$ in 
$\hporb (Y)$ are isotopic in $\hporb (Y)$ 
if and {only if} they are connected by a path 
$\phi_t$ in $\hporb (Y)$
if and only if there is $\psi$ in  
$\hpiso (Y)$ such that 
$\phi_1 = \phi_0 \circ \psi$. 

When $Y$ in Definition \ref{groupdefns} 
is one-dimensional (i.e., has a zero-di\-men\-sio\-nal cross section),  
the composants (path connected components) of $Y$ are the 
flow orbits; so, 
an element $h$ of $\hp (Y)$ is isotopic in $\hp (Y)$ to the 
identity if and only if it is an element of 
$\hporb (Y)$  which is isotopic to the identity in 
$\hporb (Y)$, {i.e., if and only if it belongs to $\hpiso (Y)$ }. 

 Below,  the image of a point $y$ under the time $t$ map of 
a flow $\gamma$ 
is denoted 
$\gamma_t(y)$ or $\gamma (y,t)$. An ambient flow may be 
denoted by $\gamma$ without comment.

We begin with a standard example. 

\begin{example} \label{dehnexample}
 Let $T$ be the identity map on the unit 
circle. The suspension flow on the 
mapping torus of $T$ can be presented as a flow on 
the 2-torus 
$\T^2 = (\R / \Z)^2$, with 
$\gamma_t:[(x,y)] \mapsto [(x,y+t)]$. The 
(``Dehn twist'') 
toral automorphism 
$h: [(x,y)] \mapsto [(x,y+x)]$ maps each 
flow orbit to itself. But, $h$ is not isotopic 
to the identity, because the homeomorphism $h$ induces a nontrivial 
automorphism of the fundamental group of 
$\T^2$.
\qed
\end{example}

In the next proposition, 
our main  interest is in the case that $Y=\maptorus X$
with $X$ zero-dimensional. 
We use $\rho$ to denote return map and $\tau$ to 
denote return time. 

\begin{theorem} 
\label{isoflowprop}
Suppose $\gamma$ is a 
flow on a compact metric space $Y$ such that 
$\gamma$ has no fixed point, and 
$h\in \hporb (Y)$. Then the following are equivalent. 
\begin{enumerate} 
\item 
There is a continuous function $\beta : 
Y\to \R $ such that \\ 
$h(y)= \gamma_{\beta(y)}(y)$, for all $y$ in  
$Y$. 
\item 
$h\in \hpiso (Y)$. 
\end{enumerate} 
Moreover, the following hold. 
\begin{enumerate}\setcounter{enumi}{2} 
\item  
Suppose $C$ and $D$ are cross sections for $Y$ such that  
$h(C)=D$, 
 and  
there is a continuous map $\beta : C\to \R$ 
such that the following hold for all $y$ in  
$C$: 
\begin{enumerate} 
\item 
$h(y)= \gamma_{\beta(y)}(y)$,   
\item 
$\beta (\rho_C(y)) - \beta (y) = \tau_D(h(y))-\tau_C(y)$. 
\end{enumerate} 
Then $(1)$ holds.
\item 
Given $ (1)$,  
for $0\leq s \leq 1$ define 
$h_s:  Y \to  Y$ by 
the rule $h_s(y) = 
\gamma_{s\beta (y)}(y) = 
\gamma (y,s\beta (y))$. 
Then $(h_s)_{0\leq s\leq 1}$ is a 
path in 
$\hpiso (Y)$ from the identity 
to $h$.
\end{enumerate} 
 \end{theorem} 

\begin{proof} 
%
%

$(1) \implies (2)${:} We will prove this implication by proving (4).
Because $\beta$ is continuous, 
$h_s$ is continuous. It remains 
to show for $0< s<1$ and 
$y \in Y$ that the restriction $h_s: \text{Orbit}(y) \to \text{Orbit}(y)$ 
is bijective and orientation preserving.

For $t\in \R$, $h_s(\gamma (y,t)) = \gamma (y, t + s\beta (\gamma
(y,t)))$. Because $s\beta $ is bounded and continuous,  
it follows that $h_s: \text{Orbit}(y) \to \text{Orbit}(y)$ is surjective. 

Suppose $\text{Orbit}(y)$ is not a circle. 
Considering $\gamma_r(y)$ and $\gamma_t (y)$, 
we see that $h_s$ is orientation preserving and injective 
on $\text{Orbit}(y)$ if and only if 
\[
r<t \implies 
r+s\beta (\gamma_r (y)) < t + s\beta (\gamma_t(y)) 
\]
which is equivalent to 
\begin{equation} \label{orbineq}
r<t \implies 
s\big(\beta (\gamma_r (y)) - \beta (\gamma_t(y)) \big) < t-r.
\end{equation} 
Because \eqref{orbineq} holds for $s=1$, it 
holds for $0<s<1$. 

Now suppose $\text{Orbit}(y)$ is a circle, 
with $p$ the smallest positive number such that 
$\gamma_p(y)=y$. The argument above, restricted 
to $r,t$ such that $0\leq r < t \leq p$, again shows 
$h_s$ is orientation preserving and injective 
on $\text{Orbit}(y)$. 

$(2)\implies (1)${:} 
 We are given a continuous 
function $H: {Y}\times [0,1] \to {Y}$, 
$(w,s)\mapsto h_s(w)$, with $h_1=h$, $ h_0=I$ and 
each $h_s \in\hporb (Y)$.  Pick $T>0$ such that 
for all $y$ the restriction of 
$\gamma$ to $ \{ y\} \times [0,4T] $ is injective. 
Appealing to uniform continuity of $H$, 
pick $\eta > 0 $ such that for any $\psi = 
h_s\circ h_r^{-1}$ with  $0\leq r < s \leq \min (1, r+ \eta )$ 
and for any $y$ in $Y$ 
there is  $c(y, \psi )$ in $(-T,T)$ such that 
$\psi (y) = \gamma (y,c(y, \psi ))$. By choice of $T$, 
the number $c(y, \psi )$ is unique, and it depends continuously 
on $y$.  Pick an integer $n> 1/\eta$. Set $\psi_0 = \text{Id}$  and for 
$1\leq i \leq n$ set 
$\psi_i =  
h_{i/n}\circ h_{(i-1)/n}^{-1}  $. 
On $Y$ define 
\[
\beta (y) =\sum_{i=1}^{n} c( \psi_{i-1}(y),\psi_i ). 
\] 
Then $\beta $ is continuous and {$h_1(y) = \gamma_{\beta (y)} (y)$}.

$(3){:}$ 
Let $c: C \times [0, \infty) \to [0, \infty )$ be the continuous function  
such that $c(x,0)=0$ and 
$h: \gamma (x,t)
\to \gamma (h(x), c(x,t) )$. 
Extend $\beta $ to all of $Y$ by defining 
$\beta$ on $U:=\{ \gamma (x,t): x\in C, 0< t < \tau_C(x)\}$ to be 
$\beta : y=\gamma(x,t) \mapsto \beta (x) + c(x,t) -t$.  By (a),  $\beta$ is
continuous on the open set $U$  
and satisfies $h(y) = \gamma_{\beta (y)} (y)$ everywhere.  
The condition (b) guarantees 
$\beta $ remains continuous on  $C$.

\end{proof}

\begin{example} \label{badsftexample}
For $Y$ the mapping torus of 
a certain reducible shift of finite type, we exhibit 
$h$ in $\hporb (Y)$
 which is not isotopic to 
the identity. 

Let $n$ be a positive integer, with $n>1$. The matrix
$A=\left(\begin{smallmatrix} 1&n\\0&1\end{smallmatrix} \right)$ 
defines an SFT $(X_A,\sigma_A)$ which consists of two fixed points 
and $n$ connecting orbits.  
Let $h$ be the homeomorphism of $X_A$ 
which acts like the shift on one connecting orbit and 
equals the identity map on the other orbits. 
Then $h$ is an automorphism of the shift 
$\sigma_A$ and induces a homeomorphism 
$\tilde h: \maptorus{X_A}\to \maptorus{X_A}$. 
Here $\tilde h(y)=\gamma_{\beta (y)}(y)$ with $\beta =1$ 
on one connecting orbit and $\beta =0$ elsewhere, 
and $\beta $ is discontinuous at the two circles in 
{$\maptorus X_A$}. If $\beta '$ is any function 
with 
 $\tilde h=\gamma_{\beta '}$, then 
$\beta' - \beta $ must 
be zero on the connecting orbits,  and 
$\beta'$ cannot be continuous.  
Therefore $h$ is not isotopic to the identity. 
\qed 


%
\end{example}

\begin{example} \label{badsoficexample}
  {We will now present} an example of a mixing sofic shift $X$ 
with an element $h$ of $\hporb ({\maptorus X})$ which is not 
isotopic to the identity.

Let $(X_A,\sigma_A)$ be the {full} two-shift on symbols $a,b$.  Let $X$ be 
the image of $X_A$ 
under the factor map 
$\pi$ which collapses the two fixed points 
$a^{\infty},b^{\infty}$ of $X_A$ 
to a single fixed point $q$ and which collapses no other points. 
The quotient system $(X,\sigma)$ is topologically conjugate to a 
  mixing sofic shift; more precisely,
  a mixing near Markov shift \cite{mbwk:amsess}.
Define a locally constant function $g$ on $X_A$ by the rule 
$g(x) = 0$ if $x_0x_1=aa$, 
$g(x) = 1$ if $x_0x_1=bb$, 
$g(x) = 1/2 $ if $x_0\neq x_1$. 
For 
$0\leq t \leq 1$ and $x\in X_A$, define  
$h: {{\maptorus{X}}_A} \to {{\maptorus{X}}_A}$ by 
\[
h:  [ (x,t) ] \mapsto [( x,  t+(1-t)g(x) + t g(\sigma_A (x) )]. 
\] 
The definitions at $[x,1]$ and $[\sigma_A( x),0]$ 
are consistent, 
and $h$ is continuous. Then $h$ is a self-equivalence of 
the flow on ${{\maptorus{X}}_A}$ because 
for $0\leq r < s \leq 1$, we have 
\begin{multline*} 
 {\big( s+ (1-s)g(x) +sg(\sigma_Ax)\big)} 
- {\big( r+ (1-r)g(x) +rg(\sigma_Ax)\big)} 
\\
= {\big( s-r \big) \big( 1 + g(\sigma_A x) - g(x)  \big)} 
 >0 
\end{multline*}  
because $|g(\sigma_A x)-g(x)|\leq 1/2$. The function 
$\beta $ defined in the notation above by 
$[ (x,t) ] \mapsto (1-t)g(x) + t g(\sigma_A (x) ) $ 
is a  continuous function on ${{\maptorus{X}}_A}$ such that 
$h(y)= \gamma (y, \beta(y) )$. 

Now define $\overline h  : {\maptorus{X}}\to {\maptorus{X}}$ by 
$\overline h (\pi (y)) = \pi  (h(y))$. 
The map $\overline h$ is well defined because 
$\pi (y) = \pi (y') \implies \pi (h(y))=\pi (h(y'))$. 
It follows that  $\overline h  \in 
\hporb ({\maptorus{X}})$. 

Suppose $\overline h \in \hpiso  ({\maptorus{X}})$. 
Then by Theorem \ref{isoflowprop} there 
is a continuous function $b 
: {\maptorus{X}} \to \R$ such that 
$\overline h (y) = \gamma (y, b(y) )$ for all $y \in {\maptorus{X}}$. 
Define $\tilde b = b \circ \pi\in C({\maptorus{X}_A}, \R)$. 
Then  $ h (y) = \gamma (y, \tilde b(y) )$ for all $y \in {\maptorus{X}}$. 
We must have  $\tilde b =  \beta$ on the set of aperiodic points;
by density of this set and continuity, we must then have
$\tilde b =  \beta$  everywhere.  
But this is impossible, since $\beta (a^{\infty}) - \beta(b^{\infty}) \neq 0$. 
This contradiction shows that $\overline h$ is not 
isotopic to the identity. 
\qed\end{example}

\section{The Parry-Sullivan {argument}} \label{sec:ps}

Theorem \ref{pstheorem} below is a formulation 
and extension of the key argument 
(in our opinion) of the Parry-Sullivan paper \cite{parrysullivan}. 
That argument  is the heart of the matter; still, 
Theorem \ref{pstheorem} adds two features to the 
content of \cite{parrysullivan}. First,  
Parry and Sullivan considered only the invertible (flow equivalence)
case of Theorem \ref{pstheorem}; 
Theorem \ref{pstheorem}
is not restricted to invertible maps
(because our study of flow equivalence via canonical covers
\cite{bce:sofic} forces us to consider noninvertible maps). 
 Second, we include
in Theorem \ref{pstheorem} an explicit statement
about isotopic triviality (not needed by Parry and Sullivan), 
with an eye to the mapping class group
of a subshift, especially an irreducible shift of finite
type.
We also give a  detailed proof of
Theorem \ref{pstheorem}, 
 to complement 
 the succinct argument in \cite{parrysullivan},
 for those of us with a less direct pipeline to
 topological truth.
 
Another version of the Parry-Sullivan  theorem (for irreducible Mar\-kov shifts, 
with a different argument and formulation) 
is given  in \cite[Section V.5]{parrytuncelbook}.  

We begin with a lemma.  
Condition $(1)$ 
in Lemma \ref{disjointify} is not needed 
to prove Theorem \ref{pstheorem}. 
The group 
$\hpiso (Y)$ in the statement  is defined 
in Definitions \ref{groupdefns}.

\begin{lemma} \label{disjointify} 
Suppose $C, C'$ are cross sections for 
 a flow $\gamma $ on  a one-dimensional space $Y$. Then there exists 
$q\in \hpiso (Y)$ such that the following hold.  
\begin{enumerate} 
\item
There is a finite subset $T$ of $\R$ such that \\ 
$q(C') \subset \{\gamma (x, t) : x\in C , t\, \in T \}$. 
\item 
{$q(C')\cap C = \emptyset$.}  
\end{enumerate} 
The homeomorphism $q$ can be chosen arbitrarily close to the identity. 
\end{lemma} 
\begin{proof}  
(2) follows from (1): 
given $q(C')$ satisfying (1) and any sufficiently small $\nu >0$, 
$(\gamma_{\nu} \circ q)(C')$ is a cross section disjoint from $C$.  
So it remains to prove (1). 

Let $\mu$ and $\mu'$ be the minimum return times to $C$ and $C'$ 
respectively under the flow. 
Suppose $0< \epsilon < \min {\{ \mu /2, \mu' /2\}}$.   

Given  $y\in C' $,  pick $v $ in   $C$ and $t \in
 \R$ such that   $y = \gamma (v, t)$.  
Because $0< \epsilon < {\mu/2}$, 
$\gamma$ maps   $ C \times {[t-\epsilon , t+ \epsilon ]}$ homeomorphically 
to a neighborhood 
of $y$.  
Pick $V' $ clopen in $C'$ and $\delta $ in $(0, \epsilon )$ 
such that 
$\gamma $ maps $V'\times (-\delta ,\delta )$ 
homeomorphically to an open neighborhood of $y$ contained in 
$\gamma (C \times {(t-\epsilon, t+\epsilon)})$.  

By compactness, we may cover $C'$ with finitely many such
neighborhoods 
$U_i'{:=}\gamma (V'_i \times {(\delta,\delta)})$, $1\leq i \leq
N$.  Suppose $i\neq j$ and $z\in U'_i \cap U'_j $.
Then there are $x_i \in V'_i, x_j \in V'_j$ and 
$\{ s_i, s_j \} \subset {(-\delta,\delta)}$ 
such that 
$z=\gamma (x_i,s_i)=\gamma (x_j,s_j)$. Then 
$\gamma (x_i , s_j -s_i ) =x_j$ with $|s_j - s_i| < {2\delta<2\epsilon} < \mu'$,  
so $s_i= s_j$ and $ x_i=x_j$. Therefore 
$U'_i \cap U'_j \subset  \gamma ((V'_i\cap V'_j)  \times {(-\delta,
\delta)})$.  
Replace each $V'_i$  with the clopen set 
$V'_i \setminus \cup_{j> i} V'_j$.  The open sets $U'_i$ are now 
pairwise disjoint but their union still covers $C'$. 

{Choose, for each $i$, a $t_i\in\R$ such that $U'_i\subset \gamma(C\times (t_i-\epsilon,t_i+\epsilon))$. Then, for each $y\in V'_i$ there is a unique $s(y)\in (-\epsilon,\epsilon)$ such that $y\in \gamma(C\times\{t_i-s(y)\})$. The function $y\mapsto s(y)$ is then continuous on $V'_i$.}
Given $s$ in ${(-\epsilon , \epsilon)}$, 
let ${\ell_{s}}$ be the  homeomorphism 
${[-\epsilon , \epsilon ] \to [-\epsilon ,\epsilon ]}$ 
which is the union of the increasing linear homeomorphisms 
${[-\epsilon , 0] \to [-\epsilon , s]}$ and 
${[0, \epsilon ]\to [s, \epsilon]}$.
With $y= \gamma (v,t) $, define 
$q: Y\to Y$ by 
{\begin{equation*}
	q(y)=
	\begin{cases}
		\gamma (v,\ell_{s(y)} (t))&
		\text{if }y=\gamma (v,t)\text{ for }(v,t)\in V'_i \times (-\epsilon , \epsilon ),\\
		& \hspace{14em} 1\leq i \leq N,\\
		y& \text{otherwise}.
	\end{cases}
\end{equation*}}

For each $i$, $q$ maps $C' \cap {U'_i}$ into $\gamma (C \times \{t_i\})$, 
so (1) holds with $T=\{ t_1, \dots , t_N \}$. 
The map $q$ is a homeomorphism mapping flow orbits to themselves 
preserving orientation. 
On each $U'_i$, with $y= \gamma (v,t)$ as above,  
$q ( y) = \gamma (y, \beta (y))$,
with $\beta (y) = {\ell_{s(y)}} (t)  -t $  continuous on $U'_i$. 
Define $\beta (y)=0$ outside $\cup_i U'_i$.
Then $\beta$ is continuous on $Y$ and 
$q ( y) = \gamma (y, \beta (y))$. 
It then follows from  Theorem \ref{isoflowprop}
that $q\in \hpiso (Y)$.

Because $|\beta(y)|< {2\epsilon} $ and 
$\epsilon >0$ was arbitrarily small, $q$ can be chosen arbitrarily 
close to the identity.
\end{proof}

\begin{theorem}[\cite{parrysullivan}] \label{pstheorem} Suppose 
$Y,Y'$ are {one-dimensional} compact metric spaces 
with fixed point free flows $\gamma, \gamma'$ 
for which $C,C'$ 
are zero-dimensional cross sections. 
Suppose 
$h: Y \to Y'$ is a flow map.  

Then there are 
discrete cross sections $D,D'$ for $\rho_C, \rho_{C'}$,  
with 
 $D\subset C$ and $D' \subset C'$ and $h^{-1}(D')=D$,  such that 
$h$ is a composition $h=h_2\circ h_1$,  
where 
\begin{enumerate} 
\item 
$h_1:Y \to Y$ lies in 
the group 
$\hpiso (Y)$,
  
\item 
$h_2: Y\to Y'$ is induced by a morphism $(D,\rho_D ) \to
(D',\rho_{D'})$. 
\end{enumerate} 
\end{theorem}

\begin{proof} By Proposition \ref{equivcrossconditions}, 
$h^{-1}(C')$ is a cross section for the flow on $Y$. 

{\it Case 1: 
{$C \cap h^{-1}(C') = \emptyset$}.} 
Because $C$ and {$h^{-1}(C')$} are disjoint cross sections, there are
continuous hitting time functions {$C \cup h^{-1}(C') \to (0, \infty )$} 
defined by  
\begin{align*} 
\tau_1 (y) &= \min \{ t\in \R : t>0, \gamma (y,t) \in C \}, \\
\tau_2 (y) &= \min \{ t\in \R : t>0, \gamma (y,t) \in  h^{-1}(C')\}.
\end{align*} 
Define cross sections $D, D'' , D'$ for $Y,Y,Y'$ (respectively) 
by 
\begin{alignat*} {2}
\{ x\in C: \tau_2 (x) < \tau_1 (x)\}  &= D \ && \subset C, \\ 
\{ \gamma (x,\tau_2(x)):  x\in {D} \}& =D''\ && \subset h^{-1} (C'), \\ 
h(D'') &=D'\ && \subset C'. 
\end{alignat*} 
Let $\rho, \rho'', \rho'$ be the return maps for the cross sections 
$D,D'' ,D'$ under the flows $\gamma , \gamma , \gamma'$ (respectively).  
Let $\psi : D \to D''$ be the homeomorphism 
$y \mapsto y''$ defined by 
$y'' = \gamma (y,\tau_2(y) )$. Given $y$ in $D$, the first four elements
(in order along the flow) 
of $ (D\cup D'')\, \cap \,\gamma (\{y\} \times [0,\infty )) $ 
 are $y, y'', \rho(y), \rho''(y'')$. 
Thus $\rho''(\psi (y)) = \psi (\rho (y))$
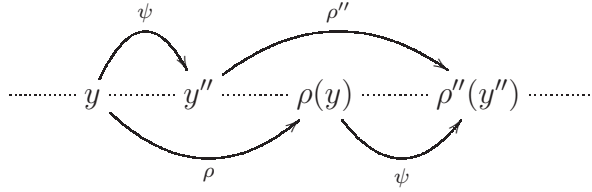
\[\xymatrix{
 \ar@{.}[r]
& 
y \ar@{.}[r]
\ar@/_2pc/[rr]_{\rho} 
\ar@/^2pc/[r]
^{\psi}  
& y'' \ar@{.}[r]
 \ar@/^2pc/[rr]
^{\rho''} 
& \rho (y)  \ar@{.}[r]
 \ar@/_2pc/[r]
_{\psi}  
&
 \rho''(y'') \ar@{.}[r]& 
}
\]
 and 
$\psi$ is a topological conjugacy 
 $(D,\rho ) \to (D'', \rho'')$.   

Let $h_1: Y \to Y$ be a flow equivalence induced 
by ${\psi^{-1}}$. 
Because $\tau_2$ is continuous on $D$, 
it follows from Theorem \ref{isoflowprop} that 
$h_1 \in \hpiso (Y)$.
Define $h_2: Y\to Y'$ by $h_2=h\circ {h_1^{-1}}$.
The map $h_2 $ restricts to a morphism $ (D,\rho) \to
(D',\rho')$    
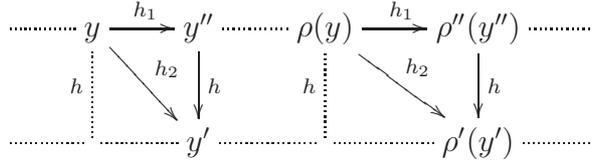
\[ \xymatrix{
 \ar@{.}[r]
& y \ar[rd]^{h_2} \ar[r]^{h_1}  \ar@{.}[d]_h
& y'' \ar@{.}[r] \ar[d]^h
& \rho (y)  \ar[rd]^{h_2}   \ar[r]^{h_1}      \ar@{.}[d]_h       
& \rho''(y'') \ar[d]^h  \ar@{.}[r]&
\\ 
 \ar@{.}[r]  
&\ar@{.}[r]& y' \ar@{.}[r]&\ar@{.}[r]&\rho'(y') \ar@{.}[r]&
} 
\]
and the flow map $h_2$ is induced by this morphism.  

{\it Case 2: 
{$C \cap h^{-1}(C') \neq \emptyset$}.} 
By Lemma \ref{disjointify} {there is a}
$q\in \hpiso (Y)$ such that $q(C)\cap {h^{-1}(C')} = \emptyset$. 
The Case 1 argument then gives cross 
sections $D\subset q(C) $ and $D'\subset {h^{-1}(C')}$ such that 
{$h=h_2\circ h_1$} where {$h_1\in \hpiso (Y)$} and $h_2$ is induced by a morphism of return maps 
for $D,D'$. Then {$q^{-1}(D)$} is a cross section; 
{$q^{-1}(D) \subset C$}; 
{$h_2\circ q$} is  an equivalence induced by  
a {morphism} of the return maps for {$q^{-1}(D)$} and $D'$; 
{$q^{-1}\circ  h_1\in \hpiso (Y)$}; and 
{$h = (h_2\circ q)\circ (q^{-1}\circ  h_1)$}. 
\end{proof}


From the flow equivalence case 
of condition (2) in Theorem \ref{pstheorem},
Parry and Sullivan \cite{parrysullivan}
and Bowen and Franks \cite{BowenFranks}
derived invariants of flow equivalence for shifts of finite
type which Franks 
\cite{jf:fesft} 
showed to be complete 
invariants for flow equivalence of nontrivial irreducible SFTs.  
For 
the Huang classification of general SFTs up to flow equivalence, 
see 
\cite{mb:fesftpf, mbdh:pbeim}. 
For the classification up to flow equivalence for 
general SFTs  with a free finite group action, 
see \cite{bce:gfe}. For partial results on the 
classification of sofic shifts up to flow 
equivalence, see \cite{bce:sofic}.

\section{Flow codes} \label{sec:flowcodes}
If $\phi :X\to Y$ is a continuous shift commuting map between subshifts, 
then $\phi$ is defined by a block code: a rule $\Phi$ defined on
$X$-words 
of length $i+j+1$ such that 
{for} all $n$ and all  $x$ in $X$, $(\phi (x))_n = \Phi ( x_{n-i}x_{n-i+1}
\cdots x_{n+j})$. We will introduce {\it flow codes} (and {\it word 
  flow codes}) to get   analogous 
invariant local codings for  flow maps.


Let $C$ be a discrete cross section for a subshift $(X,\sigma)$. 
Given $C$, the {\it return time bisequence} of a point $x$ in $C$ 
is the bisequence $(r_n)_{n\in \Z}$ 
(with $r_n=r_n(x)$) of integers such that 
\begin{enumerate}
\item 
 {$\sigma^j (x)\in C$} if and only if 
$j=r_n$ for some $n$, 
\item 
 $r_n< r_{n+1}$ for all $n$, and 
\item 
$r_0= 0$. 
\end{enumerate} 
A {\it return word} is a word equal to $x[0,r_1(x))$ for some $x\in
C$, with $x[a,b)$ denoting the finite segment $x_ax_{a+1}\cdots x_{b-1}$ of $x$. 
Given $x\in C$ and $n\in \Z$, $W_n=W_n(x)$ denotes the return 
word $x[r_n,r_{n+1})$. 
In the context of a given $C$, when we write 
$x=\dots W_{-1}W_0W_1 \dots $ below, we mean $x\in C$ and $W_n = W_n(x)$. 
Given {$x\in C$} and $i\leq j$, the tuple 
$(W_n(x))_{n=i}^j$ is the $[i,j]$ return block of $x$. 
To know this return block is to know
the word $W= W_i\cdots W_j$ together with 
its factorization as a concatenation of  return words. 

\begin{definition}\label{defnwordblockcode}
For a discrete cross section $C$ of a subshift $X$, 
a $C$  {\it word block code} 
is a function $\Phi$, which 
for some positive integer $M$ 
maps 
$[-M,M]$ return blocks 
{occurring} in $X$ 
to words. 
A  {\it word block code} 
is a $C$  {\it word block code} for some $C$. 
The  function $\phi$ from 
$C$ into a subshift 
 given by  $\Phi$ 
is defined to map  
$x = (W_n)_{n\in \Z} $   
to the concatenation 
$x' =   (W'_n)_{n\in \Z} $, with   $W'_n=  \Phi (W_{n-M},...,W_{n+M}) $  
and $x'[0,\infty )=W'_0W'_1 \dots $. {By an abuse of terminology, we will also call the function $\phi$ a word block code.}
\end{definition} 

Let $C$ and $\phi$ be as in Definition \ref{defnwordblockcode}. 
Because $C$ is clopen, there is a $\kappa \geq 0$ such that for all
$x$ in $X$, the word $x[-\kappa, \kappa]$ determines whether 
$x$ is in $C$. If $R$ is the
maximum return time to $C$, it follows for all $x$ in $C$ that  
$x[ -\kappa -MR,  MR+\kappa  ]$ determines the return block 
$(W_{-M}(x), \dots , W_{M}(x) )$.
In particular,  $\phi$  is continuous on $C$.
{Notice that} $\phi$ is not required to be injective or surjective. 


\begin{definition} \label{flowcodedef}
Let $(X,\sigma), (X',\sigma')$ be  subshifts and let  
$C$ be a discrete cross section of $\sigma$.    
A {\it $C$ word flow code} is a flow map 
$h: \maptorus X \to \maptorus X'$ 
defined 
from a $C$ word block code $\phi$ as follows: 
for each 
$ x=(W_n)_{n\in \Z} $ and 
$\phi (x)=x'=(W'_n)_{n\in \Z}$,  
 \[
h: [(x,t)] \mapsto [(x',ct)] , \quad 0\leq t \leq |W_0|,
\]
with $c= |W'_0|/|W_0|$. 
A {\it word flow code} is a 
 $C$ word flow code for some  
discrete cross section $C$. 
\end{definition} 
Let $y=[x,0]$ be a point in $\maptorus X$ with $x$ in $C$. 
We can visualize the word flow code map on the orbit of $y$ 
by adding vertical bars  to display the factorization into return words:
\begin{alignat*}{9} 
y= \cdots &\ |\ x_{-3}x_{-2}x_{-1}&&\ |\  x_0x_1   &&\ |\ x_2
&&\ |\ x_3x_4x_5&&\ |\  \cdots    \\
            &                            &&\ \ \ \ \downarrow    \\ 
h(y)=\cdots &\ |\   x'_{-2}x'_{-1} &&\ |\  x'_0x'_1x'_2 &&\ |\ x'_3x'_4 
&&\ |\ x'_5 &&\ |\  \cdots  
\end{alignat*}
The coordinates of $x'$ covered by $W'_n$ may grow arbitrarily far
from 
the coordinates $[r_n, r_{n+1})$ of $x$ covered by
$W_n$. Nevertheless, $h: \maptorus X\to \maptorus X'$ is defined by 
patching together local rules. 

Given  $\phi$  a $C$ word block code, let $X_{\phi (C)}$ be the 
subshift which is the shift closure of $\{ \phi (x) : x\in C\}$. 
Then $\phi$ defines a {word} flow code $h:\maptorus X \to 
\maptorus X_{\phi  (C)}$. However, $\phi (C)$ need not be a discrete cross section 
for $ X_{\phi  (C)}$, even if $\phi$ is 
a block code, because in general $\phi (C)$ need not be open 
as a subset of 
$X_{\phi (C)}$ (Example \ref{phiCnotopen}).  
Even if $\phi (C)$ is a discrete cross section, we would like to insist 
that $C= h^{-1} ({h(C)})$ so that 
the flow code defined from $\phi$ will be induced by a 
morphism of return maps, 
$(C,\rho_C)\to {(\phi (C), \rho_{\phi (C)})}$.
So, we will refine the definition of word flow code. 

To be completely explicit, suppose $C,C'$  are discrete 
cross sections for   
subshifts $(X,\sigma), (X',\sigma')$. For $W$ an $X$-word or $X'$
word, we use notation $[W]$ to denote the word 
 viewed as a symbol in an alphabet. For $x$ in $C$, 
let  $\eta_C$ map $x=\cdots W_{-1}W_0W_1 \cdots $ to the 
bisequence $\cdots [W_{-1}][W_0][W_1] \cdots$. Then $\eta (C)$ 
is compact and invariant under the shift map $\sigma$, and 
$\eta_C : (C, \rho_C)\to (\eta_C (C), \sigma )$ is a topological  
conjugacy. Now suppose $\psi : (C, \rho_C) \to (C', \rho_{C'})$ 
is a morphism of discrete systems. Then the map  
${\eta_{C'}\circ \psi\circ\eta_C^{-1}}
: \eta_C(C)\to \eta_{C'}(C')$ 
is a block code, 
which is used to define a $C$ word code $\Phi$ and from 
that a $C$ word flow code $h: \maptorus X\to \maptorus X'$. 
This $h$ is the word flow code induced by the morphism 
$\psi : (C, \rho_C) \to (C', \rho_{C'}) $. 

\begin{definition} \label{flowcodedefcross}
Let $(X,\sigma), (X',\sigma')$ be  subshifts with discrete cross sections 
$C,C'$.  A {\it $(C,C')$  flow code} is a 
flow map $h: {\maptorus X} \to {\maptorus X'}$ 
such that the following hold: 
\begin{enumerate} 
\item 
$h$ is a word flow code defined by a morphism 
$(C,\rho_C)\to ( C', \rho_{ C'})$. 
\item 
$h^{-1} (C') = C$ . 
\end{enumerate} 
A {\it flow code} is a 
 $(C,C')$  flow code for some $(C,C')$ as above.  
\end{definition}

\begin{example}\label{example:symbolexpansion} 
Suppose $a$ is a symbol from the alphabet $\mathcal A (X)$ of a subshift 
$(X,\sigma)$ and $a'$ is not in $\mathcal A (X)$. Define a word code 
on $X$ which copies each symbol except $a$, and maps $a$ 
to $aa'$. This is a  {\it symbol expansion}. Here $\phi (X)$ 
is a discrete cross section in $X_{\phi (X)}$, and 
the resulting $(X,\phi (X))$ flow code is a flow equivalence. 
For example, 
for $X$ the 2-shift on alphabet $\{a,b\}$  and $X'$ the golden mean
shift  on alphabet $\{ a,a',b\}$ (with $(b^*(aa')^*)^*$ 
its language), 
  this symbol expansion gives a flow equivalence 
$\maptorus X \to \maptorus X'$.  
\end{example} 

%
In Theorem \ref{isotoflowcode} below, for brevity we identify the 
base cross section $\{[(x,0)]: x\in X\} $ of ${\maptorus X}$ with $X$, and 
the return map with the shift map on $X$. 

\begin{theorem} \label{isotoflowcode}
Suppose $h: {\maptorus X}\to {\maptorus X'}$ is a flow equivalence of 
mapping tori of subshifts. Then there is a 
flow code $h_2 : {\maptorus X} \to {\maptorus X'}$  and an ${h_1\in\hpiso(\maptorus X)}$ such that $h= h_2{\circ} h_1$.  The map $h_1$ has the form 
$h_1\colon y\mapsto \gamma (y, \beta (y))$, with $\beta : Y \to \R$
continuous. 
\end{theorem} 

\begin{proof} 
As in  Theorem \ref{pstheorem}, $h= {h'_2 \circ h'_1}$, 
with ${h'_2}$ induced by a topological conjugacy of 
the return maps to discrete cross sections $C,C'$ in $X,X'$ and ${h'_1\in\hpiso(\maptorus X)}$. 
We identify the base cross sections $X,X'$ of the mapping tori 
with subshifts, with return maps given by the shift map. 
We may present these return maps as subshifts, 
with alphabets the respective return word sets. 
The topological conjugacy, as a block code with 
respect to these alphabets, is {a word block} code 
{which induces} a flow code $h_2$ {such that $h'_2=h_2\circ\phi$ for some $\phi\in \hpiso(\maptorus X)$. Let $h_1=\phi\circ h'_1$. Then $h_1\in\hpiso(\maptorus X)$ and $h= h_2\circ h_1$.} 
The form for $h_1$ follows from Theorem \ref{isoflowprop}. 
\end{proof}  

We could briefly summarize Theorem 
\ref{isotoflowcode} by saying that every 
flow equivalence  of subshifts is isotopic to 
one given by a flow code. 

Flow codes are {adapted} to 
formulating coding 
arguments for flow equivalence 
of subshifts (analogous to block codes for 
shift-commuting maps between subshifts).

\begin{remark}  
The fact that a flow equivalence 
of mapping tori of subshifts  (when it exists) is isotopic 
to a nice one (given by a flow code) is reminiscent of the 
fact 
that a continuous equivalence of smooth flows on compact hyperbolic 
sets has a $C^0$ perturbation to a H\" older continuous 
equivalence \cite[Theorem 19.1.5]{hk:book}. 
\end{remark} 

Finally we {detail} the example referred to earlier. 

\begin{example} \label{phiCnotopen} 
Let $(X,\sigma)$ be the golden mean shift 
on symbols $a_1, a_2,b$, with language $((a_1a_2)^*b^*)^*$, and 
let $\phi$ be the one block code with rule 
$a_1\mapsto a, a_2 \mapsto a, b\mapsto b$. Then 
$C=\{ x: x_0= a_1 \text{ or } x_0 =b\}$ is a discrete cross section 
of $(X,\sigma) $; $\phi (X)=X_{\phi  (C)}$; 
the fixed point $a^{\infty }$ is in $\phi (C) $; and 
 $\phi (C) $ contains no $\phi (X)$ neighborhood of 
$a^{\infty }$ (because for every positive integer $n$, $\phi (C)$ contains
no point $x$ such that $x[0,2n+1]= a^{2n+1}b$. 
 Therefore $\phi (C)$ is not open 
(alternately, note the return time to $\phi (C)$ is not continuous). 
Therefore $\phi (C)$ is not a discrete cross section for the subshift 
$X_{\phi  (C)}$. 
\end{example}

\section{Flow equivalence induced by conjugacy}
\label{floweqconjsec} 

We consider conditions on a flow equivalence of mapping tori 
which can force it to 
be induced by a topological conjugacy of  their bases.  

For completeness, we first recall an old, simple lemma (for which we do not know the original reference). Given an edge $e$ in a directed graph, we let $\iota (e)$ be its initial vertex and $\tau (e)$ its terminal vertex. By a cycle of edges, we mean a string of edges $e_0\dots e_k$ such that $\tau(e_i) =\iota (e_i)$ for each $i$, and $\tau(e_k) =\iota (e_0)$.

  \begin{lemma} \label{graphlemma}
  Suppose $f$ is a function from edges of an irreducible finite directed graph
  $\mathcal G$ into an abelian group $G$, such that for every cycle of edges
  $e_0\dots e_k$ in the graph, $\sum_{0\leq i \leq k}f(e_i)=0$.
  Then there is a function $h$ from vertices into $G$ such that for all $e$, 
  $f(e) = h(\tau (e)) - h(\iota (e))$.
  \end{lemma}
  \begin{proof}
  Pick a vertex $v_0$ in the graph. Define $h(v_0)=0$. Given a vertex $v$,
  by irreducibility there exists a path  
  $e_0\dots e_k$  with $\iota (e_0)= v_0$ and $\tau (e_k)=v$. 
  Define $h(v)= \sum_{0\leq i\leq k} f(e_i)$.  By the
  zero-on-cycles assumption, $h(v)$ does not depend on the path
  (it must be the negative of the sum of $h$ along any path from $v_0$
  back to $v$).  
  \end{proof}

\begin{proposition} \label{sftlivsic} 
Suppose $f$ is a locally constant function 
from an irreducible SFT $(X,\sigma)$ into an abelian group, and $f$ 
sums  to zero over each 
periodic orbit. 
Then  there is a locally constant 
function $b $ such that 
$f= b \circ \sigma - b$. 
\end{proposition}
\begin{proof}
After passing to a higher block 
presentation,  without loss of generality one can assume
the SFT is an edge shift and 
 $f(x)$ depends only on the edge $x_0$,  $f(x)=f(x_0)$.
Let $h$ be as in Lemma \ref{graphlemma} and define $b(x)=h(\iota (x_0))$. Then
$f(x)=b(\sigma (x)) - b(x)$.
\end{proof}
For a generalization of Proposition \ref{sftlivsic}
to nonabelian groups (with a more subtle conclusion), 
see \cite[Theorem 9.3]{parrylivsic} and  
\cite{schmidtlivsic}; for a remark on that conclusion,
see \cite[Remark 4.7]{BoSc2}.
Proposition \ref{sftlivsic} is a special case of a subshift  
 version of the Liv\v sic Theorem \cite[Thm. 19.2.1]{hk:book}.

{The next result  is another subshift version of a theorem 
  for smooth hyperbolic flows
  (see \cite[Theorem 19.2.8]{hk:book}).} 
\begin{theorem} \label{isftfebyconj}
Suppose that there are flows on $Y, Y'$ having cross sections
$C, C'$ such that their return maps
$(C, \rho_C), (C', \rho_{C'})$ are topologically conjugate to irreducible SFTs
$(X,\sigma), (X',\sigma'),$ respectively.
Suppose that $h: Y\longrightarrow Y'$ is an equivalence such that
for each circle $\mathcal{C}$ in $Y$,
$|{\mathcal{C}} \cap C| =  |h({\mathcal{C}}) \cap C'|$.
Then  
$h$  is isotopic to an equivalence 
 induced by a topological 
conjugacy of the given SFTs, $X\to  X'$. 
\end{theorem}

\begin{proof} 
Without loss of generality, we assume $Y=\maptorus{X}$, $Y'=\maptorus{X'}$ 
with their standard unit speed flows. 
After passage to an isotopic map, we may also assume that $h$ 
is given by a flow code. So, there is a clopen subset $C$ of 
$X$ which is a discrete cross section for $\sigma$; and if {$\sigma^i(x)\in C$} 
and $r$ is the return time of {$\sigma^i(x)$} to $C$, then 
$h$ replaces the word $x[i, i+r)$ with some word $W=W_0\cdots W_{s-1}$ 
depending continuously on $x$. The corresponding 
orbit interval of length $r$ in $\maptorus{X}$ is sent by a linear 
time change ($t\mapsto \frac sr t$) to an orbit interval of length $s$. 

Now, for  $x \in X$ there is a minimal positive 
number $\ell (x)$ such that there is  $r\in \Z$ 
such that the  map $h$ takes the 
orbit segment $\{ [(x,t)]: 0\leq t <1 \}$ bijectively to 
an orbit segment
$\{ [(y,s)]: r\leq s < r+ \ell (x)\}$. Because 
$h$ is given by a flow code, the function $\ell$ is 
locally constant on $X$. 
 Define on $X$ the continuous function 
$g(x) = \ell(x)-1$. 

If $x$ has least period 
$n$ for $\sigma$, then 
$h$ maps the circle through $[(x,0)]$  homeomorphically 
to a circle of equal length, and therefore we have
$\sum_{i=0}^{n-1} g(\sigma^i (x)) =0$.
{By Proposition \ref{sftlivsic},
 there is a locally constant 
function $b $ such that 
$g= b \circ \sigma - b$.}
Define {$\beta  :  \maptorus X\to \R$} 
by setting, for $x\in X$ and $0\leq t \leq 1$, 
\[
\beta \colon  [(x,t)]\mapsto 
b(x) + t(b(\sigma x) -b(x)) 
=b(x) +t \ell (x)\ . 
\]  
Note the definitions at  
$[(x,1)]$ and {$[(\sigma(x),0)]$} are
consistent. 
Define {$k : \maptorus X\to \maptorus X_A$} 
by $k : y\mapsto \gamma (y, \beta (y))$.  
The map $k$ sends each flow line in $\maptorus X $ 
bijectively to itself, and the function $\beta$ is 
continuous. 
It follows from Theorem 
\ref{isoflowprop} that $k$ is isotopic to the identity, 
and therefore $j:=h\circ k^{-1}$ is isotopic to 
$h$.  

Next we check that 
the map $j\colon \maptorus X \to \maptorus X'$ is a 
conjugacy of suspension flows. 
For $x\in X$, 
remembering {$k([(x,0)])=[(x,b(x))]$}, we have
\[ 
\xymatrix{
{k([(x,0)])} & & {[(\sigma(x),b(\sigma (x)))]} \ar@{=}[r]& 
 \gamma \big( k([x,0]) ,\ell (x) \big) \qquad \  \\
[(x,0)]\ar[u]^k \ar[d]_h&  & 
{[(\sigma(x),0)]}\ar[u]^k \ar[d]_h \ar@{=}[r] &
 \gamma \big([(x,0)], 1\big) \qquad \qquad \  \\ 
h([(x,0)]) & & 
{h([(\sigma(x),0)])} \ar@{=}[r] & 
\gamma' \big(h( [(x,0)]  ) , \ell (x) \big) 
}
\] 
and for $0 \leq  t  \leq 1$, we have 
\begin{align*} 
k \colon & [(x,t)] \mapsto 
\gamma {\big( k( [(x,0)]) , t\ell (x)\big)} \\ 
h \colon & [(x,t)] \mapsto 
\gamma {\big( h( [(x,0)]) , t\ell (x)\big)}. 
\end{align*}  
Consequently, 
$\gamma'\circ  j = j\circ\gamma$  at $\gamma(x,t)$ 
for all $x\in X$ and $0\leq t \leq 1$; hence 
$j$ is a conjugacy of suspension flows.  

Finally, for a point $x$ 
with dense orbit in $X$, let $\tau \in \R $ be such that 
$k$ takes $[(x,0)]$ to $[(x',\tau )]$, with $x'\in X'$ ($\tau$ is unique 
if $X$ is not just a single finite orbit.) Then for every $i$ in $\Z$, 
$\gamma_{-\tau} \circ j$ takes {$[\sigma^i(x),0]$} 
to {$[\sigma'^i(x'),0]$}. By density of the orbit in $X$, 
$\gamma_{-\tau} \circ j$ takes 
$X \times \{ 0\}$ to $X' \times \{ 0\}$, 
and therefore defines 
 a topological {conjugacy} of $\sigma$ and $\sigma'$
(given the obvious identification of $\sigma: X\to X$ with 
the return map to $X\times \{0\}$ under the suspension 
flow).

 \end{proof}

\begin{theorem} \label{isftorbittoiso}
Suppose 
$h\in \hporb ({\maptorus X)}$, 
with 
$X$ an irreducible SFT.  
Then 
$h\in \hpiso {(\maptorus X)}$. 
\end{theorem} 

\blue{\begin{remark} \label{minimalcase}
Theorem \ref{isftorbittoiso} is also known to hold 
when $X$ is a minimal subshift 
(see \cite[Theorem 2.5]{app:shgtl} and its references). 
\end{remark}}

\begin{proof}[Proof of Theorem \ref{isftorbittoiso}]   
By Proposition \ref{isftfebyconj}, 
$h$ is isotopic to an equivalence induced by 
an automorphism of an irreducible SFT fixing every periodic orbit. 
Such an automorphism can only be a power of the shift 
\cite{bk:firstauto},
so $h$ is isotopic to the identity.  
(For an alternate proof of the corollary, take $\tau$ in the proof 
of Theorem \ref{isftfebyconj} such that $j$ maps $[(x,0)]$ to
itself.)
\end{proof} 

\begin{example} \label{soficsameorbitcounter} 
Theorem \ref{isftorbittoiso} fails for reducible SFTs  
(Example \ref{badsftexample}) and 
mixing sofic shifts  (Example \ref{badsoficexample}). 
Theorem \ref{isftfebyconj} 
also 
  becomes false if the assumption ``irreducible SFT''
  is replaced with ``irreducible sofic'' (or with 
``mixing {sofic}''). 
To see this, note that 
$h: \maptorus X \to \maptorus X$ 
in Example \ref{badsoficexample} 
takes each orbit of $\maptorus X$ into itself. 
Suppose this $h$ is induced by an automorphism $U$ 
of $X$. 
As a  factor map, $\pi$  is conjugate to the Fischer cover
of a sofic shift; therefore, there is a unique automorphism 
$\widetilde{U}$  of $X_A$ which is a lift of 
 $U$ \cite{KriegerSoficI}. Because $\widetilde{U}$ must fix 
all periodic orbits of $X_A$, except perhaps the two fixed points,  
$\widetilde{U}$ must be a power of the shift 
\cite{bk:firstauto}. Therefore  $U$ is a power of the shift, 
and therefore $h$ is isotopic to the identity. But this 
contradicts  
the conclusion of Example \ref{badsoficexample}. 
  \end{example} 

\begin{remark}
We are considering fixed-point-free flows on a mapping torus 
$Y={\maptorus X}$, with $\gamma_t $ the time $t$ map of 
the flow. 
That $X$ is a cross section to the flow means that there is a surjective 
local homeomorphism $\pi: X \times \R\to Y$ such that 
$\pi (x,t) = \gamma_t (x) $.  
In this context of considering isotopy of maps, one might hope that for a 
flow equivalence  
$h: {\maptorus X} \to {\maptorus X'}$, there might be a 
homeomorphism $\tilde h: X \times \R \to X \times \R$ such 
that $\tilde h{\circ} \pi = \pi{\circ} h$. In general, no such lift exists. 
\end{remark} 

\begin{proposition} 
\label{lifteqconj}
Suppose for $i=1,2$ that $T_i: X_i\to X_i$ is a homeomorphism 
of a zero-dimensional  compact metric space. Let $\gamma_i$ be the 
suspension flow on the 
mapping 
torus $\maptorus X_i$. 
 Let $\pi_i: X_i \times \R \to \maptorus X_i$ be the map  
$(x,t)\mapsto \gamma_i (x,t)$. 
Suppose $h: \maptorus X_1 \to \maptorus X_2$ 
is a topological equivalence of 
the suspension flows.  

Then the following are equivalent. 
\begin{enumerate} 
\item 
$h$ is  isotopic in $\hporb (Y)$
to an equivalence of flows 
$Y_1\to Y_2$ induced by a 
topological conjugacy of $(X_1,T_1)$ and  $(X_2,T_2)$. 
\item 
There is a homeomorphism $\widetilde h: X_1\times \R \to X_2\times \R$ 
such that $h{\circ}\pi_1= \pi_2{\circ} \widetilde h$. 
\end{enumerate} 
\end{proposition} 
\begin{proof} 
$(1)\implies (2){:}$ 
From the  isotopy, there is 
a $j\in \hpiso (\maptorus X_1)$ such that  
$h\circ j$ is a conjugacy of suspension 
flows sending the cross section $\pi_1 (X_1\times \{0\} )$ onto
$\pi_2 (X_2\times \{0\} )$. Let $k\colon X_1 \to X_2$  be the 
homeomorphism such that for $x$ in $X_1$, 
$h\circ j \colon \pi_1 (x,0) \to \pi_2(k(x),0)$. 
Then the map $X_1\times \R \to X_2 \times \R$ 
defined by $(x,t)\mapsto (k(x),t)$ is a lift of 
$h\circ j$. 
It now suffices to check that there is a lift $\widetilde j$ of $j$.  
By Theorem \ref{isoflowprop}, 
there is a continuous function $\beta : \maptorus X_1 \to \mathbb R$ 
such that for all $x$ in $X_1$ and $t\in \mathbb R$, 
 $j \colon [(x,t)] \mapsto \gamma ([(x,t)], \beta (x,t) )$. 
Define $\widetilde j \colon X_1\times \R \to X_1\times \R$ by 
$\widetilde j \colon (x,t)\mapsto (x, t+ \beta (\gamma_t (x)) $. 
Then $\widetilde j$ is continuous; 
$j{\circ}\pi_1= \pi_1{\circ} \widetilde j$; for each $x\in X_1$, 
$\widetilde j$ is bijective on  $x\times \R$; and   
$\widetilde j$ is a homeomorphism.

%
%
$(2) \implies (1){:}$ 
The given lift $\widetilde h: \ X_1 \times \R\ \to \  X_2 \times \R $ 
has the form $\widetilde h: (x,t) \mapsto ( 
\overline{h}(x), t + \beta (x,t) )$, 
with  
$\beta $  continuous 
and $\overline{h}: X_1 \to X_2$ a homeomorphism. 
Because $\pi_1$ is a local homeomorphism and 
$h{\circ}\pi_1= \pi_2{\circ} \widetilde h$, it follows that $\beta $ 
is the lift of a continuous function (also denoted $\beta$) on 
$\maptorus X_1$: for all $x \in X_1$ and all $t$, 
$h\colon  \gamma (x,t) \mapsto [ ( \overline h(x) , t+\beta (x,t) ) ]$. 
Define $j: \maptorus X_1 \to  \maptorus X_1 $ by $y \mapsto 
\gamma (y,\beta (y) )  $.  
Because $\widetilde h$ is orientation preserving, for every 
$x$ in $X$ we have  
\[
s<t \implies  s+ \beta (x,s) < t+\beta (x,t).
\] 
Consequently $j \in \hp  (\maptorus X_1)$, and by 
Theorem \ref{isoflowprop} it follows that 
$j \in \hpiso (\maptorus X_1)$. 
Therefore 
$h$ is  isotopic in $\hporb (Y)$ to 
$h\circ j^{-1}$, which  
lifts to the map 
$ (x,t) \mapsto (\overline h (x), t)$.  
Therefore the equivalence 
$h\circ j^{-1}$ is a conjugacy of flows induced by 
the 
topological conjugacy  $(X_1,T_1) \to (X_2,T_2)$  
defined by $\overline h$. 
\end{proof} 
\begin{remark} 
If in Proposition \ref{lifteqconj}(2)  we only require $\widetilde h$ to 
be continuous instead of a homeomorphism, then 
$\widetilde h$ can easily be constructed: first define on $X_1 \times
\{0\}$, then extend. This $\widetilde h$ need 
not be 
surjective and need not be injective (even if $X_1=X_2$). 
However, $\pi_2\circ
\widetilde h$ will be surjective. 
\end{remark}

\section{Extending equivalences and cross sections}
\label{sec:extension}
We finish
with a pair of extension results for flows with zero-dim\-en\-sio\-nal cross sections.
The definition of 
$\hpiso (Y)$ is in 
{Definition} \ref{groupdefns}. 

\begin{prop}\label{extendisos}
Suppose a flow on a compact metric space $Y$ has a zero-dimensional 
cross section. 
Suppose $Y'$ in $Y$ is the domain of 
a subflow of $Y$ and  $\chi \in \hpiso (Y')$. 
Then $\chi$ extends to 
$\widetilde{\chi}:Y\to Y$ such that 
$\widetilde{\chi} \in \hpiso (Y)$. 
\end{prop}

\begin{proof} 
Without loss of generality, 
we assume $Y$ is the mapping torus $\maptorus{X}$ 
of a homeomorphism $T:X\to X$ of {a zero-dimensional 
compact metric space} $X$.

By Theorem \ref{isoflowprop}, 
there is a continuous map  $\phi: Y' \to\R$ such that 
\begin{equation*}
	\chi(y)=\gamma_{\phi(y)}(y)
\end{equation*}	
for $y\in Y'$.  
By Theorem \ref{isoflowprop},  it suffices to extend $\phi$ to a continuous map
$\tilde{\phi}:\maptorus{X}\to\R$ such that 
the function $\tilde{\chi} : x\mapsto\gamma_{\tilde{\phi}(x)}(x)$ 
maps each orbit of $\maptorus{X}$ to itself 
by an orientation preserving  homeomorphism. 

Let $X'=\{x\in X:[x,0]\in Y'\}$. Then $X'$ is 
$T$-invariant and $Y'
=\{[y,t]:x\in X',\ t\in\R\}$. Since $\chi$ is orientation preserving,
we have for all $x$ in $X'$ that  
\begin{equation*}
	1+\phi([x,1])-\phi([x,0])>0.
\end{equation*}	
By the continuity of $\phi$ and compactness of $Y'$, 
there is a $\mu > 0$ such  that 
\begin{equation*}
	\mu =\min\{1+\phi([x,1])-\phi([x,0]):x\in X'\}.
\end{equation*}
Let $\mathcal P_n$, $n \in \mathbb N$, be a sequence of 
nested finite partitions of $X$ into clopen sets, with 
$\text{diameter} (C) < 1/n$ for all $C$ in $\mathcal P_n$. 
Let $\mathcal Q_n=\{C\in \mathcal P_n \colon 
C\cap X' \neq \emptyset\}$ .  Pick $N$ such that 
for any pair $x,z$ from $X'$, 
\[
\text{dist} (x,z) < 1/N 
\implies 
|\phi ([x,0]) -\phi ([z,0]) | < \mu /3. 
\] 
From each $C$ in $\mathcal Q_n$ with $n\geq N$, pick a point 
$z_{n,C}$ in $X'$. 
Let $Q$ be the clopen set in $X$ which is the union of the $C$ 
in $\mathcal Q_N$. 
Define  continuous maps $\pi \colon Q \to X'$ 
and $\alpha : Q \to \mathbb R$ by 
\begin{align*} 
\pi (x) &= 
\begin{cases}
	x&\text{if } x\in X', \\
	z_{n,C}&\text{if } x\in C \in \mathcal Q_n 
\text{ and } x\notin \bigcup_{C\in Q_{n+1}} C, 
\end{cases}
\\ 
\alpha (x) &= \phi ([\pi x,0] ).
\end{align*} 
Note, if both $x$ and $T(x)$ are in $Q$, then 
\begin{align*}
1+\alpha &(T(x)) - \alpha (x) \ 
=\  1 + \phi ( \pi (T(x))-\phi ( \pi (x))\\ 
=&\ 1+\big( \phi ( \pi (T(x))-\phi ( T (x))\big)\\
&\quad + \big(\phi ( T (x)) -\phi (x) \big)
+\big( \phi(x) -\phi ( \pi (x)) \big) 
\\ 
 >& \ \big(1+\phi ( T (x)) -\phi (x) \big) -2\frac{\mu }{3} 
\ >\ \frac{\mu }{3}
%
%
%
\end{align*} 
and therefore
\begin{align}\label{positivedifference}  
 1+&\alpha (T(x)) - \alpha (x) > 0 .
\end{align} 
%
%
Choose an integer $M$ such that 
\begin{equation} \label{Mchoice}
M>\max\{|\alpha(x)|:x\in Q \}  .
\end{equation}
Define clopen subsets of $X$, 
\begin{align*} 
Q^{(M)} \ &=\ \bigcap_{-M\leq n \leq M} T^nQ 
\\
W\ & =\ \{ x\in X: |n|<M \implies 
T^n(x)\notin Q^{(M)} \} \\
V\  &  =\  W\cup Q^{(M)} . 	
\end{align*} 
Extend $\alpha$ to 
a continuous function on $V$ by setting 
$\alpha(w)=0$ for $w\in W$.

The set $C:=\{[v,0]:v\in V \}$ 
is a cross section for $\maptorus{X}$. For $v\in V$, let $\tau_C(v)$
be the return time of $[v,0]$ to $C$ under the flow, and let $R_C(v)$
be the unique element $v'\in V$ for which
$[v',0]=\gamma_{\tau_C(v)}([v,0])$. 
We make the Claim: for all $v$ in $V$, 
\begin{equation} \label{posorbclaim}
\tau_C(v) + \alpha (R_C(v)) -\alpha (v)  >0  . 
\end{equation} 
We check the claim by four cases.  

{\it Case I:} $v$ and $R_C(v)$ belong to $Q^{(M)}$. \\
Here 
$\tau_C(v) + \alpha (R_C(v)) -\alpha (v)  
=1 +\alpha (T(v)) -{\alpha (v)>0}$ by
\eqref{posorbclaim}. 

{\it Case II :} $v\in W$ and $R_C(v)\in W$. \\ 
Here $\tau_C(v) + \alpha (R_C(v)) -\alpha (v)  
=1 +0-0>0$ . 

{\it Case
  III :} $v\in Q^{(M)}$ and $R_C(v) \in W$. \\ 
This is the case that 
$v\in Q^{(M)}$; 
 $T^M(v) \notin Q$; and  $0<n<M \implies T^n(v) \in Q$. 
Then 
$\tau_C(v) + \alpha (R_C(v)) -\alpha (v)  = 
M  + 0 - \alpha (v) $, which is positive 
by \eqref{Mchoice}.  

{\it Case
  IV :} $v\in  W$ and $R_C(v) \in Q^{(M)}$. \\ 
The argument is very similar to that of Case III. 
The Claim is proved.

From here, the notation $[v,s]$ refers to $v\in V$ and 
$0\leq s \leq \tau_C(v)$.  
We define $\widetilde{\chi}:Y\to Y$ by setting
\[
\widetilde{\chi}: [v,s] \mapsto 
[v, \alpha (v) + \frac s{\tau_C(v)} 
\big( \tau_C (v) + \alpha (R_C(v)) - \alpha (v) \big) ]  . 
\] 
The definition is consistent because 
$[v,\tau_C(v)]\mapsto \big[ [v,\tau_C(v)] ,  \alpha (R_C(v) ) \big]
$ 
is in agreement with 
$[R_C(v),0] \mapsto [R_C(v), \alpha (R_C(v))]$. 
The map $\widetilde{\chi}$ is continuous. 
It then follows from  the Claim that 
$\widetilde{\chi}$ sends each orbit to itself by a  
map which is piecewise 
(hence globally) 
an orientation preserving homeomorphism. Finally, 
define 
\[
\widetilde{\phi}:[v,s] \mapsto 
\Bigg(\frac {s}{\tau_C(v)}\Bigg) \alpha (R_C(v)) + 
\Bigg(1-\frac {s}{\tau_C(v)} \Bigg) \alpha (v)   . 
\] 
Then $\widetilde{\phi}$ is continuous on $Y$ and 
\[
\widetilde{\chi}: [v,s] \mapsto 
\gamma \big([v,s], \widetilde{\phi} ( [v,s] ) \big)  .
\] 
This completes the proof that 
$\widetilde{\chi}  \in  \hpiso (Y)$.  
\end{proof}

We do not know if 
the assumption of a zero-dimensional cross section 
in Proposition \ref{extendisos} is necessary. 

\begin{proposition}\label{extendsec}
Suppose $\gamma $ is a flow on a one 
dimensional compact metric space $Y$; 
$Y'$ is compact and invariant in $Y$;  
and $C'$ is a cross section for  
the subflow on $Y'$. 
Then there is a cross section $C$ for $Y$ 
such that $C\cap Y' =C'$. 
\end{proposition}

\begin{proof}  
Let $X$ be a zero-dimensional cross section
for the flow on the one-dimensional space $Y$ 
(Proposition \ref{zerocross}).
Then $X'= Y' \cap X$ is a cross section 
for the subflow on $Y'$. For simplicity and 
without loss of generality, suppose $Y$ is 
the mapping torus $\maptorus X$. 

By Lemma \ref{disjointify}, 
there exists a finite subset $T$ of $\R$ and 
$q\in \hpiso (Y)$ such that 
$q(C') \subset \{\gamma (x, t) : x\in X' , t\in T \}$. 
For $t\in T$, define $D'_t = \{ x\in X' : \gamma (x,t) \in q(C')\}$ 
and $D' = \cup_t D'_t$.  
Each $D'_t$ is clopen in $X'$.
Choose clopen subsets $D_t$ of $X$ such that 
$D_t\cap X' = D'_t$; then 
\[
Y' \cap (\cup_t \gamma (D_t \times \{t\}) ) 
= Y' \cap q(C') .
\] 
Let $D_X = \cup_{t\in T}D_t$, a clopen subset of $X$.  
Let $R$ be the maximum return time under $\gamma'$  to  
$D'$ and set 
$E=X \cap \gamma (D_X\times [0,R])$. 
  Then 
$E$ is clopen in $X$  
and $Y' \cap E=\emptyset$.  
  Define 
\[
D = \gamma (E\times \{ 0\}) \cup (\cup_t \gamma (  D_t\times \{t \}) ) . 
\] 
Then $D$ is a cross section for the flow on $Y$ and 
$D\cap Y' =
q(C')$. 

Let $\chi = q^{-1} \in \hpiso (Y')$.  
By Proposition \ref{extendisos}, 
$\chi $ extends to $\widetilde{\chi}$ in 
$\hpiso (Y)$. 
 Now 
$\widetilde{\chi} (D)$ is a cross section 
for the flow on $Y$   
 and 
$ Y' \cap \widetilde{\chi} (D) = \widetilde{\chi} (Y' \cap D)  
= {\widetilde{\chi} (q(C'))}  = C'$. 
 \end{proof}

%






\mbox{}\\ \small
\noindent
\textsc{Department of Mathematics, 
University of Maryland,
College Park, MD 20742-4015, USA}

\emph{E-mail address: }\texttt{mmb@math.umd.edu}\\[0.5cm]

\noindent
{\textsc{Department of Science and Technology, University of the Faroe Islands, N\'oat\'un 3, FO-100 T\'orshavn, the Faroe Islands}}

{\emph{E-mail address: }\texttt{toke.carlsen@gmail.com}}\\[0.5cm]

\noindent
\textsc{Department of Mathematical Sciences,  University of Copenhagen, DK-2100 Copenhagen \O, Denmark}

\emph{E-mail address: }\texttt{eilers@math.ku.dk}\\

\end{document}